\documentclass{birkjour}
 \newtheorem{theorem}{Theorem}[section]
 \newtheorem{Corollary}[theorem]{Corollary}
 \newtheorem{lemma}[theorem]{Lemma}
 \newtheorem{proposition}[theorem]{Proposition}
 \theoremstyle{definition}
 \newtheorem{definition}[theorem]{Definition}
 \theoremstyle{remark}

 \numberwithin{equation}{section}

\begin{document}

%
%
%
%
%
%
%
%
%

\title[Quaternion Kaehler-like statistical manifolds]
 {Basic inequality for statistical submanifolds of quaternion Kaehler-like statistical manifolds}

\author[M.S. Lone]{Mohamd Saleem Lone}

\address{%
International Centre for Theoretical Sciences, \\
Tata Institute of Fundamental Research,\\
 560089, Bengaluru, India.}

\email{saleemraja2008@gmail.com, \\mohamdsaleem.lone@icts.res.in}

\author[M.A. Lone]{Mehraj Ahmad Lone}

\address{%
	Department of Mathematics,\\National Institute of Technology,\\ Hazratbal, 190006, Srinagar}

\email{mehrajlone@nitsri.net}




\date{March 30, 2017}

\begin{abstract}
In the present paper, we obtain the basic Chen inequalities for submanifolds of quaternion Kaehler-like statistical manifolds. Also, we discuss the same inequality for Lagrangian submanifolds.
\end{abstract}

\maketitle
\section{Introduction}
A statistical manifold is a natural geometric generalization of statistical model. In 1985 Amari introduced the notion of statistical manifolds via information geometry (see \cite{amari2012differential}). These manifolds are equipped with dual connections (torsion-free), an analogue to conjugate connections in affine geometry (see \cite{nomizu1994affine}). Since dual connections are not metric, it is very difficult to introduce a notion of sectional curvature using the canonical definitions of Riemannian geometry. In this connection B. Opozda in \cite{opozda2016sectional} showed a way to define a sectional curvature tensor on a statistical manifold. The definition of a statistical  manifold is motived from the statistical model in a way such that Riemannian manifold $M$ replaces the density functions of a statistical model, a Riemannian metric replaces a Fisher information matrix, the pair of dual connections $(\hat{\nabla}, \hat{\nabla}^\ast)$ replaces the dual connections $(\nabla^{(-1)}, \nabla^{(1)})$, and a 3-covariant skewness tensor replaces a skewness tensor. While studying the geometric properties of a submanifold, a very important problem is to obtain sharp relations between the intrinsic and the extrinsic invariants and a vast number of such relations are revealed by certain inequalities. For example, let $M$ be a surface in Euclidean 3-space, we know the Euler inequality: $K\leq |H|^2$, where $H$ is the mean curvature (extrinsic property) and $K$ is the Gaussian curvature (intrinsic property). The equality holds at points where  $M$ is congruent to an open piece of a plane or a sphere (umbilical points). B.-Y Chen \cite{chen1996mean} obtained the same  inequality for submanifolds of real space forms. Then in \cite{chen1999relations}, B.-Y Chen obtained the Chen-Ricci inequality, which is a sharp relation between the squared mean curvature and the Ricci curvature of a Riemannian submanifold of a real space form.

In recent time, the statistical manifolds are being examined very actively. For some of the recent works, we refer \cite{nomizu1994affine,opozda2016sectional,chen2019chen,takano2004statistical}. The chronology behind the motivation of this paper that we are going the follow is the following: Takano \cite{takano2004statistical} published a few papers on statistical manifolds  with almost complex and almost contact structure. In 2015, A.D. V{\^\i}lcu and G.E. V{\^\i}lcu \cite{vilcu2015statistical} studied statistical manifolds with quaternionic settings and  proposed several open problems. While answering one of those open problems, M. Aquib \cite{aquib2019some} obtained some of the curvature properties of submanifolds and a couple of inequalities for totally real statistical submanifolds of quaternionic Kaehler-like statistical space forms. Recently, B.-Y Chen et al. \cite{chen2019chen} derived a Chen first inequality for statistical submanifolds in Hessian manifolds of constant Hessian curvature. Following the same paper, H. Atimur et al. obtained the same inequalities for statistical submanifolds of Kaehler-like statistical manifolds. Motivated from the above discussed papers, we take into consideration the quaternion Kaehler-like statistical manifolds and obtain interesting inequalities.

The structure of this paper is as follows. In the second section, we collect some definitions and lemmas which are helpful to prove the main results of the paper. In the third section we prove the main result of the paper and in the last section, we discuss the $\delta(2,2)$ inequalities.

\section{Preliminaries}
Let $(\hat{M}, \hat{g})$ be a Riemannian manifold with a pair of torsion free affine connections $\hat{\nabla}$ and $\hat{\nabla}^*$. Then $(\hat{\nabla}, \hat{g})$ is called statistical structure on $(\hat{M})$ if
\begin{eqnarray*}
	(\hat{\nabla}_{X}\hat{g})(Y,Z) - (\hat{\nabla}_{Y}\hat{g})(X,Z) = 0
\end{eqnarray*}
for $X,Y,Z \in T\hat{M}$. A Riemannian manifold $(\hat{M}, \hat{g})$ with statistical structure satisfying the compatibility condition
\begin{eqnarray*}
	X\hat{g}(Y,Z) = \hat{g}(\hat{\nabla}_{X}Y ,Z) +  \hat{g}(Y , \hat{\nabla}^*_{X}Z)
\end{eqnarray*}
is said to be a statistical manifold and is denoted as $(\hat{M},\hat{g}, \hat{\nabla}, \hat{\nabla}^*)$.
Any torsion-free connection $\hat{\nabla}$ has a dual connection $\hat{\nabla}^*$ and satisfy
\begin{eqnarray*}
	\hat{\nabla}^\circ = \frac{\hat{\nabla} + \hat{\nabla}^*}{2},
\end{eqnarray*}
where $\hat{\nabla}^\circ$ is the Levi-Civita connection on $\hat{M}$.

The curvature tensor fields  with respect to dual connections $\hat{\nabla}$ and $\hat{\nabla}^*$ are denoted by $\hat{R}$ and  $\hat{R^*}$. The curvature tensor field $\hat{R}^\circ$  associated with the $\hat{\nabla}^\circ$ is called Riemannian curvature tensor.

Generically, the dual connections are not metric, one cannot define the sectional curvature is statistical settings as in the case of Riemannian geometry. A notable difference here is that while writing the curvature (sectional), in addition to the contribution showed by dual connection via $R^\ast$, there is a correction term in the form of $R^\circ$ contributed by the Levi-Civita connection $\nabla^\circ.$  In this connection Opozda  proposed two notions of sectional curvature on  statistical manifolds (see \cite{opozda2016sectional,opozda2015bochner}).

Let $\hat{M}$ be a statistical manifold and $\pi$ a plane section in $T\hat{M}$ with orthonormal basis $\{X,Y\}$, then the sectional $K$-curvature  is defined in \cite{opozda2016sectional} as 
\begin{eqnarray*}
	\hat{K}(\pi) = \frac{1}{2}\bigg[\hat{g}(\hat{R}(X,Y)Y,X) + \hat{g}(\hat{R}^*(X,Y)Y,X) - \hat{g}(\hat{R}^\circ (X,Y)Y,X)\bigg].
\end{eqnarray*}

The curvature tensors $\hat{R}$ and  $\hat{R^*}$ satisfy the following property

\begin{eqnarray*}
	\hat{g}(\hat{R}(X,Y)Y,W)  = -\hat{g}(\hat{R}^*(X,Y)W,Z).
\end{eqnarray*}

Let $\hat{M}$ be a differentiable manifold and assume that there is a rank 3-bundle $\Lambda$ of $End(T\hat{M})$, such that a local basis $\{ J_\alpha\}$ exists on the section of $\Lambda$  satisfying
$$J_{\alpha}^2 = -I, J_{\alpha} J_{\alpha +1} = - J_{\alpha+1} J_{\alpha} = J_{\alpha+2},$$
where $\{\alpha=1,2,3\}$ and $I$ is identity tensor field of type $(1,1)$ on $\hat{M}$. The indices are being taken from $\{1,2,3\}$ modulo $3$. In this case $\{J_{1}, J_{2},J_{3}\}$ is called canonical basis of $\Lambda$ and $\Lambda$ is called almost quaternion structure on $\hat{M}$. Moreover, ($\hat{M}, \Lambda$) is called almost quaternionic with dimension $4m,$ $ n\geq 1$.

A Riemannian metric $g$ on $\hat{M}$ is said to be adapted to the almost quaternionic structure $\Lambda$ if it satisfies  
\begin{equation}\label{k5} g(J_{\alpha}X, J_{\alpha}Y) = g(X,Y), \alpha \in\{1,2,3\}\end{equation}
for all vector fields $X,Y$ on $\hat{M}$ and any canonical basis $\{J_1,J_2,J_3\}$ of $\Lambda.$
\begin{definition}
	Let $(\hat{M},g)$ be a Riemannian manifold with an almost quaternion structure $\Lambda$ having $\{J_1,J_2,J_3\}$ its canonical basis with $\{J_1^\ast,J_2^\ast,J_3^\ast\}$ three other tensor fields of type $(1,1)$ satisfying
	\begin{equation}
		g(J_\alpha X,Y)+g(X,J_\alpha^\ast Y)=0
	\end{equation} 
	for all vector fields $X,Y$ on $\hat{M}.$ Then $(\hat{M},\Lambda,g)$ is called as almost Hermite-like quaternion manifold and if this manifold is endowed with the torsion free and symmetric connection pair $(\nabla,\nabla^\ast)$, then $(\hat{M,\nabla,\Lambda,g})$ is called as almost Hermite-like quaternion statistical manifold. If this $J^\ast$ satisfy (\ref{k5}), then we can consider a subbundle of $End(T\hat{M})$ locally spanned by $\{J_1^\ast , J_2^\ast , J_3^\ast\}$, such that we have
	$$(J_\alpha^\ast)^\ast=J_\alpha$$
	and $$g(J_\alpha X, J_\alpha^\ast Y)=g(X,Y)$$
	for all vector fields $X,Y$ on $\hat{M}$ and $\alpha \in \{1,2,3\}$.
\end{definition}
\begin{definition}
	Let $(\hat{M}, \hat{\nabla}, \Lambda, g)$ be an almost Hermite-like quaternionic  statistical manifold. Then $(\hat{M}, \hat{\nabla}, \Lambda, g)$ is said to be quaternionic Kaehler-like statistical manifold if for any local basis $\{J_{1},J_{2},J_{3}\}$ of $\Lambda$  there exist three locally defined $1$-form $\omega_{1},$ $\omega_{2},$ $\omega_{3}$ on $\hat{M}$ such that we have 
	$$(\hat{\nabla}_{X}J_{\alpha})Y = \omega_{\alpha+2}(X)J_{\alpha +1}Y - \omega_{\alpha+1}(X)J_{\alpha+2}Y$$
	for all vector fields on $\hat{M}$ and $\alpha \in \{1,2,3\}$.
\end{definition}

\begin{definition}
	Let $(\hat{M}, \hat{\nabla}, \Lambda, g)$ be a  quaternionic Kaehler-like statistical manifold. Then the curvature tensor $R$ with respect to $\hat{\nabla}$ satisfies
	\begin{eqnarray}{\label{curvaturetensor}}
		\hat{R}(X,Y)Z\nonumber &=& \frac{c}{4}\bigg\{g(Y,Z)X - g(X,Z)Y\\
	\nonumber	&& + \sum_{\alpha=1}^{3}\bigg[g(Z,J_{\alpha}Y)J_{\alpha}X - g(Z,J_{\alpha}X)J_{\alpha}Y\bigg]\\
		&&+ \sum_{\alpha =1}^{3}\bigg[g(X, J_{\alpha}Y)J_{\alpha}Z - g(J_{\alpha}X,Y)J_{\alpha}Z\bigg]\bigg\} 
	\end{eqnarray}
	for all vector fields $X,Y,Z$ on $\hat{M}$ and $c$ is a real constant.
\end{definition}
The curvature tensor $\hat{R}^{*}$ with respect to dual connection $\hat{\nabla}^{*}$ is obtained just by replacing $J_{\alpha}$ by $J_{\alpha}^{*}$.\\
Let  $(M, g, \nabla, \nabla^*)$ be statistical submanifold of $(\hat{M}, \hat{g}, \hat{\nabla}, \hat{\nabla}^*)$. The Gauss and Weingarten formulae are given as
\begin{eqnarray*}
	&&\hat{\nabla}_{X}Y = \nabla_{X}Y + \sigma(X,Y) ,\quad  \hat{\nabla}_{X}\xi = -A_{\xi}X + \nabla_{X}^{\perp}\xi\\
	&&\hat{\nabla}^*_{X}Y = \nabla^*_{X}Y + \sigma^*(X,Y) , \quad \hat{\nabla}^*_{X}\xi = -A^*_{\xi}X + \nabla_{X}^{*\perp}\xi
\end{eqnarray*}
for all $X, Y \in TM$ and $\xi \in T^{\perp}M$ respectively. Moreover, we have the following equations
\begin{eqnarray*}
	&&Xg(Y,Z) = g(\nabla_{X}Y, Z) + g(Y, \nabla^{*}_{X}Z) \\
	&&\hat{g}(\sigma(X,Y), \xi) = g(A^*_{\xi}X,Y) , \quad{}  \hat{g}(\sigma^*(X,Y), \xi) = g(A_{\xi}X,Y)\\
	&&   X\hat{g}(\xi, \eta) = \hat{g}(\nabla_{X}^{\perp}\xi, \eta) + \hat{g}(\xi, \nabla_{X}^{*\perp}\eta).
\end{eqnarray*}
The mean curvature vector fields for orthonormal tangent and normal frames $\{e_{1}, e_{2}, \dots,e_{n}\}$ and $\{e_{n+1}, e_{n+2}, \dots,e_{4m}\}$, respectively, are defined as
\begin{equation*}
	H =\frac{1}{n}\sum_{i=1}^{n}\sigma(e_{i},e_{i}) = \frac{1}{n}\sum_{\gamma=1}^{4m}\left( \sum_{i=1}^{n}\sigma_{ii}^{\gamma}\right)\xi_{\gamma},\text{ } \sigma_{ij}^{\gamma} = g(\sigma(e_{i},e_{j}), e_{\gamma}) 
\end{equation*}
and \begin{equation*}  H^* =\frac{1}{n}\sum_{i=1}^{n}\sigma^{*}(e_{i},e_{i}) = \frac{1}{n}\sum_{\gamma=1}^{4m}\left( \sum_{i=1}^{n}\sigma_{ii}^{*\gamma}\right)\xi_{\gamma}, \text{ } \sigma_{ij}^{*\gamma} = g(\sigma^*(e_{i},e_{j}), e_{\gamma})
\end{equation*}
for $1\leq i,j \leq n$ and $1\leq l\leq 4m$.

Now, we state the following fundamental results on statistical manifolds.
\begin{proposition}\cite{vilcu2015statistical}
	Let  $(M, g, \nabla, \nabla^*)$ be statistical submanifold of $(\hat{M}, \hat{g}, \hat{\nabla}, \hat{\nabla}^*).$ Let $\hat{R}$ and $\hat{R}^*$ be the Riemannian curvature tensors on $\hat{M}$ for $\hat{\nabla}$ and $\hat{\nabla}^*$, respectively. Then the Gauss, Codazzi and Ricci equations are given by the following result.
	\begin{eqnarray*}
		\hat{g}(\hat{R}(X,Y)Z,W) &=& g(R(X,Y)Z,W) + \hat{g}(\sigma(X,Z), \sigma^*(Y,W)) \\
		&&-  \hat{g}(\sigma^*(X,W), \sigma(Y,Z)),
	\end{eqnarray*}	
	\begin{eqnarray*}
		\nonumber \hat{g}(\hat{R}^*(X,Y)Z,W) &=& g(R^*(X,Y)Z,W) + \hat{g}(\sigma^*(X,Z), \sigma(Y,W))\\
		&& -  \hat{g}(\sigma(X,W), \sigma^*(Y,Z)),
	\end{eqnarray*}	
	\begin{equation*}
		\hat{g}(R^\perp(X,Y)\xi, \eta) = \hat{g}(\hat{R}(X,Y)\xi, \eta) + g([A_{\xi}^*, A_{\eta}]X,Y),
	\end{equation*}	
	\begin{equation*}
		\hat{g}({R^*}^\perp(X,Y)\xi, \eta) = \hat{g}(\hat{R}^*(X,Y)\xi, \eta) + g([A_{\xi}, A_{\eta}^*]X,Y),
	\end{equation*}	
	where $[A_{\xi}, A_{\eta}^*] = A_{\xi}A_{\eta}^* - A_{\eta}^*A_{\xi}$ and $[A_{\xi}^*, A_{\eta}] = A_{\xi}^*A_{\eta} - A_{\eta}A_{\xi}^*$, for $X, Y, Z, W \in TM$ and $\xi, \eta \in T^\perp M$.
\end{proposition}
Now, we state two important lemma's which we  use to prove the main results in the upcoming sections.
\begin{lemma}\label{lemma1}
	Let $n \geq 3$ be an integer and $a_1, a_2,\dots ,a_n$ are $n$ real numbers. Then, we have 
	$$\sum_{1\leq i<j \leq n}^n a_i a_j -a_1 a_2 \leq \frac{n-2}{2(n-2)}\bigg(\sum_{i=1}^n a_i\bigg)^2.$$
\end{lemma}
\begin{lemma}\label{lemma2}
	Let $n \geq 4$ be an integer and $a_1, a_2,\dots ,a_n$ are $n$ real numbers. Then, we have 
	$$\sum_{1\leq i<j \leq n}^n a_i a_j -a_1 a_2 - a_{3}a_{4} \leq \frac{n-3}{2(n-2)}\bigg(\sum_{i=1}^n a_i\bigg)^2.$$
\end{lemma}

\section{A first  Chen inequality}
In the present section, we obtain first Chen inequality for quaternion Kaehler-like statistical manifolds.
\begin{theorem}
	Let $(\hat{M}, \hat{g}, \hat{\nabla}, J)$ be a quaternion Kaehler-like statistical manifold of dimension $4m$ and $M$ be its statistical submanifold of dimension $n$:\\
	(a) If $M$ is a holomorphic statistical submanifold, then we have
	\begin{eqnarray*}
		& &	(\tau - \tau_{\circ}) - (K(\pi)- K_{\circ}(\pi)) \\ &&\geq \frac{c}{8}(n-2)(n+1) + \frac{c}{4}\bigg\{ \sum_{\alpha=1}^{3}\frac{1}{2}(tr P_{\alpha})^2 + \|P_{\alpha}\|^{2} -2 tr(P_{\alpha}P^{*}_\alpha)\bigg\} \\ & &-\frac{1}{2}\sum_{\alpha=1}^{3}g^2(e_{1}, J_{\alpha}e_{2})  -\frac{1}{2}\sum_{\alpha=1}^{3}g^2(J_{\alpha}e_{1}, e_{2}) - \frac{1}{2}\sum_{\alpha=1}^{3}g(e_{1}, J_{\alpha}e_{2})g(e_{2}, J_{\alpha}e_{2})\\ & & - \frac{n^{2}(n-2)}{4(n-1)}\big[\|H\|^2 + \|H^*\|^2\big] + 2\hat{K}_{\circ}(\pi) -2\hat{\tau}_{\circ}.
	\end{eqnarray*}
	(b) If $M$ is a real statistical manifold, then we have
	\begin{eqnarray*}
		& &	(\tau - \tau_{\circ}) - (K(\pi)- K_{\circ}(\pi)) \\ &&\geq \frac{c}{8}(n-2)(n-1)  - \frac{n^{2}(n-2)}{4(n-1)}\big[\|H\|^2 + \|H^*\|^2\big] + 2\hat{K}_{\circ}(\pi) -2\hat{\tau}_{\circ}.
	\end{eqnarray*}
	Moreover, the equalities holds  for any $\gamma \in \{n+1, n+2, \dots, 4m\}$ if and only if
	$ \sigma_{11}^{\gamma} + \sigma_{22}^{\gamma} = \sigma_{33}^{\gamma} = \dots = \sigma_{nn}^{\gamma}$\\
	$ \sigma_{11}^{*\gamma} + \sigma_{22}^{*\gamma} = \sigma_{33}^{*\gamma}= \dots =\sigma_{nn}^{*\gamma} $\\
	$ \sigma_{ij}^{\gamma} = \sigma_{ij}^{*\gamma}= 0$ $\forall 1\leq i\neq j \leq n$.
\end{theorem}
\begin{proof}
	Let $\{e_{1},e_{2}, \dots, e_{n}\}$ and $\{e_{n+1}, e_{n+2},\dots, e_{4m}\}$ be the orthonormal frames of $TM$ and $T^{\perp}M$, respectively.
	Then the  sectional $K$-curvature $K(\pi)$ of the plane section $\phi$ is
	\begin{eqnarray}\label{p1}
		K(\pi) &=& \frac{1}{2}\bigg[g(R(e_{1},e_{2})e_{2}),e_{1}) + g(R^{*}(e_{1},e_{2})e_{2}),e_{1}) -2g(R^{\circ}(e_{1},e_{2})e_{2}),e_{1})\bigg]. \nonumber\\
	\end{eqnarray}
	Using (\ref{curvaturetensor}) and Gauss equation for $R$ and $R^*$ and put the values in (\ref{p1}), we arrive at
	\begin{eqnarray*}
		K(\pi)& = &\frac{c}{4}\bigg\{1+ \frac{1}{2}\sum_{\alpha = 1}^{3}g^2(e_{1}, J_{\alpha}e_{2}) + \frac{1}{2}\sum_{\alpha = 1}^{3}g^2(J_{\alpha}e_{1}, e_{2})   \\ 
&&	+ \sum_{\alpha = 1}^{3}g(J_{\alpha}e_{2}, e_{2})g(J_{\alpha}e_{1}, e_{1})	-2 \sum_{\alpha = 1}^{3}g(e_{1}, J_{\alpha}e_{2})g(e_{2}, J_{\alpha}e_{1})\}- K_{\circ}(\pi)\\
&& + \frac{1}{2}\sum_{\gamma=n+1}^{4m}\bigg[\sigma_{11}^{\gamma}\sigma_{22}^{*\gamma} + \sigma_{11}^{*\gamma}\sigma_{22}^{\gamma} - 2 \sigma_{12}^{*\gamma}\sigma_{12}^{\gamma} \bigg] \bigg\}.
	\end{eqnarray*}
	Using $\sigma +\sigma^* = 2\sigma^{\circ}$, we get
	\begin{eqnarray*}
		K(\pi) &=& \frac{c}{4}\bigg\{1+ \frac{1}{2}\sum_{\alpha = 1}^{3}g^2(e_{1}, J_{\alpha}e_{2}) + \frac{1}{2}\sum_{\alpha = 1}^{3}g^2(J_{\alpha}e_{1}, e_{2}) \\ 
		& & + \sum_{\alpha = 1}^{3}g(J_{\alpha}e_{2}, e_{2})g(J_{\alpha}e_{1}, e_{1})  -2 \sum_{\alpha = 1}^{3}g(e_{1}, J_{\alpha}e_{2})g(e_{2}, J_{\alpha}e_{1})\} - K_{\circ}(\pi)\\
		&& + 2\sum_{\gamma=n+1}^{4m}\bigg[\sigma_{11}^{\circ\gamma}\sigma_{22}^{\circ\gamma} - (\sigma_{12}^{\circ\gamma})^2\bigg] -\frac{1}{2}\sum_{\gamma=n+1}^{4m}\bigg\{\bigg[ \sigma_{11}^{\gamma}\sigma_{22}^{\gamma} -  (\sigma_{12}^{\gamma})^2 \bigg]\\
		&& + \bigg[\sigma_{11}^{*\gamma}\sigma_{22}^{*\gamma}- (\sigma_{12}^{*\gamma})^2 \bigg]\bigg\}.
	\end{eqnarray*}
	Using Gauss equation with respect to Levi-Civita connection, we have 
	\begin{eqnarray}\label{m1}
		K(\pi)\nonumber &=& K_{\circ}(\pi) + \frac{c}{4}\bigg\{1+ \frac{1}{2}\sum_{\alpha = 1}^{3}g^2(e_{1}, J_{\alpha}e_{2}) + \frac{1}{2}\sum_{\alpha = 1}^{3}g^2(J_{\alpha}e_{1}, e_{2})  \nonumber \\ & &  + \sum_{\alpha = 1}^{3}g(J_{\alpha}e_{2}, e_{2})g(J_{\alpha}e_{1}, e_{1}) -2 \sum_{\alpha = 1}^{3}g(e_{1}, J_{\alpha}e_{2})g(e_{2}, J_{\alpha}e_{1})\}  \nonumber \\ & & - 2\hat{K}_{\circ}(\pi) -\frac{1}{2}\sum_{\gamma=n+1}^{4m}\bigg[ \sigma_{11}^{\gamma}\sigma_{22}^{\gamma} -  (\sigma_{12}^{\gamma})^2 \bigg] \nonumber \\
		&&-\frac{1}{2}\sum_{\gamma=n+1}^{4m}\bigg[ \sigma_{11}^{*\gamma}\sigma_{22}^{*\gamma} -  (\sigma_{12}^{*\gamma})^2 \bigg].
	\end{eqnarray}
	The scalar curvature corresponding to the sectional $K$-curvature is 
	\begin{eqnarray*}\label{p2}
		\tau = \frac{1}{2}\sum_{1\leq i<j\leq n}\bigg[g(R(e_{i},e_{j})e_{j}),e_{i}) + g(R^{*}(e_{i},e_{j})e_{j}),e_{i}) -2g(R^{\circ}(e_{i},e_{j})e_{j}),e_{i})\bigg].
	\end{eqnarray*}
	Using (\ref{curvaturetensor}) and  Gauss equation for $R$ and $R^*$. After doing some simple calculations, we  get
	\begin{eqnarray*}
		\tau &=& \frac{c}{8}n(n-1) + \frac{c}{4}\sum_{1\leq i<j\leq n}\bigg\{\frac{1}{2}\sum_{\alpha = 1}^{3}g^2(e_{i}, J_{\alpha}e_{j}) + \frac{1}{2}\sum_{\alpha = 1}^{3}g^2(J_{\alpha}e_{i}, e_{j})  \\ & & + \sum_{\alpha = 1}^{3}g(J_{\alpha}e_{j}, e_{j})g(J_{\alpha}e_{i}, e_{i}) -2 \sum_{\alpha = 1}^{3}g(e_{i}, J_{\alpha}e_{j})g(e_{j}, J_{\alpha}e_{i})\bigg\}-\tau_{\circ}\\
		&& -  \frac{1}{2}\sum_{\gamma=n+1}^{4m}\sum_{1\leq i<j\leq n}\bigg[\sigma_{ii}^{*\gamma}\sigma_{jj}^{*\gamma} + \sigma_{ii}^{\gamma}\sigma_{jj}^{*\gamma} - 2 \sigma_{ij}^{*\gamma}\sigma_{ij}^{\gamma} \bigg].
	\end{eqnarray*}
	By using some fundamental notations, the last equation  reduces to 
	\begin{eqnarray*}
		\tau &=& \frac{c}{8}n(n-1) + \frac{c}{4}\sum_{1\leq i<j\leq n}\bigg\{\|P\|^2 - \frac{1}{2} (tr P_{\alpha})^2 - tr(P^2_{\alpha})\bigg\}-\tau_{\circ}\\ & & -  \frac{1}{2}\sum_{\gamma=n+1}^{4m}\sum_{1\leq i<j\leq n}\bigg[\sigma_{ii}^{*\gamma}\sigma_{jj}^{*\gamma} + \sigma_{ii}^{\gamma}\sigma_{jj}^{*\gamma} - 2 \sigma_{ij}^{*\gamma}\sigma_{ij}^{\gamma} \bigg],
	\end{eqnarray*}
	which can be written as 
	\begin{eqnarray*}
		\tau &=& \frac{c}{8}n(n-1) + \frac{c}{4}\sum_{1\leq i<j\leq n}\bigg\{\|P\|^2 - \frac{1}{2} (tr P_{\alpha})^2 - tr(P^2_{\alpha})\bigg\}-\tau_{\circ}\\ & & -  2\sum_{\gamma=n+1}^{4m}\sum_{1\leq i<j\leq n}\bigg[\sigma_{ii}^{\circ\gamma}\sigma_{jj}^{\circ\gamma} - (\sigma_{ij}^{\circ\gamma})^2\bigg]\\ & & -\frac{1}{2}\sum_{\gamma=n+1}^{4m}\sum_{1\leq i<j\leq n}\bigg\{\bigg[ \sigma_{ii}^{\gamma}\sigma_{jj}^{\gamma} -  (\sigma_{ij}^{\gamma})^2 \bigg] + \bigg[\sigma_{ii}^{*\gamma}\sigma_{jj}^{*\gamma}- (\sigma_{ij}^{*\gamma})^2 \bigg]\bigg\}.
	\end{eqnarray*}	
	By using Gauss equation for the Levi-civita connection,we have 
	\begin{eqnarray}\label{m2}
		\tau \nonumber &=& \tau_{\circ} + \frac{c}{8}n(n-1) + \frac{c}{4}\bigg\{\|P\|^2 - \frac{1}{2} (tr P_{\alpha})^2 - tr(P^2_{\alpha})\bigg\}-2\hat{\tau}_{\circ} \\ & & -\frac{1}{2}\sum_{\gamma=n+1}^{4m}\sum_{1\leq i<j\leq n}\bigg[ \sigma_{ii}^{\gamma}\sigma_{jj}^{*\gamma} -  (\sigma_{ij}^{\gamma})^2 \bigg] \nonumber \\ &&-\frac{1}{2}\sum_{\gamma=n+1}^{4m}\sum_{1\leq i<j\leq n}\bigg[ \sigma_{ii}^{*\gamma}\sigma_{jj}^{*\gamma} -  (\sigma_{ij}^{*\gamma})^2 \bigg].
	\end{eqnarray}
	From (\ref{m1}) and (\ref{m2}), we have 
	\begin{eqnarray*}
		&&	(\tau - K(\pi)) - (\tau_{\circ}- k_{\circ}(\pi)) = \frac{c}{8}(n-2)(n-1) + \frac{c}{4}\bigg\{ \sum_{\alpha=1}^{3}(tr P_{\alpha})^2 + \|P_{\alpha}\|^{2}  \\ & &-2 tr(P_{\alpha}P^{*}_\alpha)\bigg\}-\frac{1}{2}\sum_{\alpha=1}^{3}g^2(e_{1}, J_{\alpha}e_{2})  -\frac{1}{2}\sum_{\alpha=1}^{3}g^2(J_{\alpha}e_{1}, e_{2}) \\
		&&- \frac{1}{2}\sum_{\alpha=1}^{3}g(e_{1}, J_{\alpha}e_{2})g(e_{2}, J_{\alpha}e_{2}) + \frac{1}{2}\sum_{\alpha=1}^{3}g(J_{\alpha}e_{1}, e_{2})g(J_{\alpha}e_{2}, e_{1})\bigg\}  \\ && -\frac{1}{2}\sum_{\gamma=n+1}^{4m}\bigg[ \sigma_{ii}^{\gamma}\sigma_{jj}^{*\gamma}-  (\sigma_{ij}^{\gamma})^2 \bigg] -\frac{1}{2}\sum_{\gamma=n+1}^{4m}\bigg[ \sigma_{ii}^{*\gamma}\sigma_{jj}^{*\gamma} -  (\sigma_{ij}^{*\gamma})^2 \bigg] \\ && + \frac{1}{2}\sum_{\gamma=n+1}^{4m}\sum_{\alpha=1}^{3}\bigg\{\bigg[\sigma_{11}^{\gamma}\sigma_{22}^{\gamma} -(\sigma_{12}^{\gamma})^2\bigg] + \bigg[\sigma_{11}^{*\gamma}\sigma_{22}^{*\gamma} - (\sigma_{12}^{*\gamma})^2\bigg]\bigg\}\\
		&& + 2\hat{K}_{\circ}(\pi) -2\hat{\tau}_{\circ}.
	\end{eqnarray*}
	Using  lemma \ref{lemma1}, we can get the above equation in simplied form as
	\begin{eqnarray*}
		&&	(\tau - K(\pi)) - (\tau_{\circ}- k_{\circ}(\pi))\geq \frac{c}{8}(n-2)(n-1) + \frac{c}{4}\bigg\{ \sum_{\alpha=1}^{3}(tr P_{\alpha})^2 + \|P_{\alpha}\|^{2}  \\ & &-2 tr(P_{\alpha}P^{*}_\alpha)\bigg\}-\frac{1}{2}\sum_{\alpha=1}^{3}g^2(e_{1}, J_{\alpha}e_{2})  -\frac{1}{2}\sum_{\alpha=1}^{3}g^2(J_{\alpha}e_{1}, e_{2})\\ & &  - \frac{1}{2}\sum_{\alpha=1}^{3}g(e_{1}, J_{\alpha}e_{2})g(e_{2}, J_{\alpha}e_{2})- \frac{n^{2}(n-2)}{4(n-1)}\big[\|H\|^2 + \|H^*\|^2\big] \\
		&&+ 2\hat{K}_{\circ}(\pi) -2\hat{\tau}_{\circ}.
	\end{eqnarray*}
	This proves our claims.
\end{proof}
\begin{Corollary}
	Let $(\hat{M}, \hat{g}, \hat{\nabla}, J)$ be a quaternion Kaehler-like statistical manifold of dimension $4m$ and $M$ be a totally real statistical submanifold of dimension $n$, if for any $p\in M$ and $\tau \in T_{p}M$ such that
	\begin{eqnarray*}
		& &	(\tau - \tau_{\circ}) - (K(\pi)- K_{\circ}(\pi)) < \frac{c}{8}(n-2)(n-1) + 2\hat{K}_{\circ}(\pi) -2\hat{\tau}_{\circ},
	\end{eqnarray*}
	then $M$ is non minimal.
\end{Corollary}
\begin{theorem}
	Let $(\hat{M}, \hat{g}, \hat{\nabla}, J)$ be a quaternion Kaehler-like statistical manifold of dimension $4m$ and $M$ be a Lagrangian statistical submanifold of dimension $n$. If $n\geq 4$ and $M$ satisfies the equality case of the Chen's first inequality, then it is minimal, i.e., $H=H^\ast =0.$
\end{theorem}
\section{$\delta(2,2)$ Inequality}
Let $p\in M$, $\pi_{1}, \pi_{2} \subset T_p M$ be mutually orthogonal planes spanned by $\{e_1,e_2\}$ and $\{e_3,e_4\}$, respectively. Also let $\{e_1,e_2,\cdots, e_n\}$ and $\{e_{n+1},e_{n+2},\cdots,e_{4m}\}$ be the orthonormal basis of  $T_pM$ and $T_p^\perp M$, respectively.

By doing simple calculations for $K(\pi_1)$ and $K(\pi_2)$ and using lemma \ref{lemma2}, we can obtain the following inequality.

\begin{eqnarray*}
	&&	(\tau - K(\pi_1)-K(\pi_2) - (\tau_{\circ} - k_{\circ}(\pi_1)-k_{\circ}(\pi_2)) \\&& \geq \frac{c}{8}(n^2-n-4) + \frac{c}{4}\bigg\{ \sum_{\alpha=1}^{3}(tr P_{\alpha})^2 + \|P_{\alpha}\|^{2} -2 tr(P_{\alpha}P^{*}_\alpha)\bigg\} \\ & &-\frac{1}{2}\sum_{\alpha=1}^{3}g^2(e_{1}, J_{\alpha}e_{2})  -\frac{1}{2}\sum_{\alpha=1}^{3}g^2(J_{\alpha}e_{1}, e_{2}) - \frac{1}{2}\sum_{\alpha=1}^{3}g(e_{1}, J_{\alpha}e_{2})g(e_{2}, J_{\alpha}e_{2})\\
	&&-\frac{1}{2}\sum_{\alpha=1}^{3}g^2(e_{3}, J_{\alpha}e_{4})  -\frac{1}{2}\sum_{\alpha=1}^{3}g^2(J_{\alpha}e_{3}, e_{4}) - \frac{1}{2}\sum_{\alpha=1}^{3}g(e_{3}, J_{\alpha}e_{3})g(e_{4}, J_{\alpha}e_{3})\\
	& & - \frac{n^{2}(n-2)}{4(n-1)}\big[\|H\|^2 + \|H^*\|^2\big] + 2\hat{K}_{\circ}(\pi) -2\hat{\tau}_{\circ},
\end{eqnarray*}
which represents the Chen $\delta(2,2)$ inequality for arbitrary submanifold in a quaternionic Kaehler-like statistical manifold. Using this inequality, we have the following theorem. 
\begin{theorem}
	Let $(\hat{M}, \hat{g}, \hat{\nabla}, J)$ be a quaternion Kaehler-like statistical manifold of dimension $4m$ and $M$ be its statistical submanifold of dimension $n$:\\
	(a) If $M$ be a holomorphic statistical submanifold, then we have
	\begin{eqnarray*}
		&&	(\tau - K(\pi_1)-K(\pi_2) - (\tau_{\circ} - k_{\circ}(\pi_1)-k_{\circ}(\pi_2)) \\&& \geq 
		\frac{c}{8}(n^2-n-4) + \frac{c}{8}\big(tr J_\alpha)^2 - \frac{n^{2}(n-2)}{4(n-1)}\big[\|H\|^2 + \|H^*\|^2\big] \\ & &-\frac{1}{2}\sum_{\alpha=1}^{3}g^2(e_{1}, J_{\alpha}e_{2})  -\frac{1}{2}\sum_{\alpha=1}^{3}g^2(J_{\alpha}e_{1}, e_{2}) - \frac{1}{2}\sum_{\alpha=1}^{3}g(e_{1}, J_{\alpha}e_{2})g(e_{2}, J_{\alpha}e_{2})\\
		&&-\frac{1}{2}\sum_{\alpha=1}^{3}g^2(e_{3}, J_{\alpha}e_{4})  -\frac{1}{2}\sum_{\alpha=1}^{3}g^2(J_{\alpha}e_{3}, e_{4}) - \frac{1}{2}\sum_{\alpha=1}^{3}g(e_{3}, J_{\alpha}e_{3})g(e_{4}, J_{\alpha}e_{3})\\
		& & - 2\big(\hat{\tau}_{\circ} - 2\hat{K}_{\circ}(\pi_{1}) - 2\hat{K}_{\circ}(\pi_{2})).
	\end{eqnarray*}
	(b) If $M$ be a real statistical manifold, then we have
	\begin{eqnarray*}
		&&	(\tau - K(\pi_1)-K(\pi_2) - (\tau_{\circ} - k_{\circ}(\pi_1)-k_{\circ}(\pi_2)) \\&& \geq
		\frac{c}{8}(n^2-n-4)  - \frac{n^{2}(n-2)}{4(n-1)}\big[\|H\|^2 + \|H^*\|^2\big] \\ & &- 2\big(\hat{\tau}_{\circ} - 2\hat{K}_{\circ}(\pi_{1}) - 2\hat{K}_{\circ}(\pi_{2})). 
	\end{eqnarray*}
	Moreover, the equalities holds  for any $\gamma \in \{n+1, n+2, \dots, 4m\}$ if and only if
	$ \sigma_{11}^{\gamma} + \sigma_{22}^{\gamma} = \sigma_{33}^{\gamma} + \sigma_{44}^{\gamma} = \sigma_{55}^{\gamma} \dots = \sigma_{nn}^{\gamma}$\\
	$ \sigma_{11}^{*\gamma} + \sigma_{22}^{*\gamma} = \sigma_{33}^{*\gamma} + \sigma_{44}^{*\gamma}= \sigma_{55}^{*\gamma} \dots =\sigma_{nn}^{*\gamma} $\\
	$ \sigma_{ij}^{\gamma} = \sigma_{ij}^{*\gamma}= 0$ $\forall \text{ }1\leq i\neq j \leq n.$
\end{theorem}
\begin{Corollary}
	Let $(\hat{M},\hat{g}, \hat{\nabla}, J)$ be a quaternion Kaehler-like statistical manifold of dimension $4m$ and $M$ be a totally real statistical submanifold of dimension $n$, if for any $p\in M$ and $\tau \in T_{p}M$ such that
	\begin{eqnarray*}
		& &	(\tau - \tau_{\circ})<  (K(\pi_1)- K_{\circ}(\pi_1))+(K(\pi_2)- K_{\circ}(\pi_2))+ (n^2-n-4)\frac{c}{8}\\
		&& - 2(\hat{\tau}_{\circ}-\hat{K}_{\circ}(\pi_1)-\hat{K}_{\circ}(\pi_2)),
	\end{eqnarray*}
	then $M$ is non minimal.
\end{Corollary}
\begin{theorem}
	Let $(\hat{M}, \hat{g}, \hat{\nabla}, J)$ be a quaternion Kaehler-like statistical manifold of dimension $4m$ and $M$ be a Lagrangian statistical submanifold of dimension $n$. If $n\geq 6$ and $M$ satisfies the equality case of the Chen $\delta(2,2)$ inequality, then it is minimal, i.e., $H=H^\ast =0.$
\end{theorem}
\section*{Acknowledgement} A major part of the paper was discussed and written while the second author was hosted by the first author at International Centre for Theoretical Sciences, TIFR, Bengaluru.
\section*{Conflict of interest} On behalf of both the authors, the corresponding author states that there is no conflict of interest. 
\bibliography{references} 
\bibliographystyle{acm.bst}
\end{document}